\documentclass[runningheads,a4paper]{llncs}
\usepackage[utf8]{inputenc}
\usepackage{hyperref}
\usepackage{verbatim}
\usepackage{mathtools}
\usepackage{caption,color}
\usepackage{amssymb,amstext,amsmath}
\newcommand{\overbar}[1]{\mkern 1.5mu\overline{\mkern-1.5mu#1\mkern-1.5mu}\mkern 1.5mu}
\newcommand{\no}[1]{\overbar{#1}}

\def\F{\mathcal F}

\def\M{\mathcal M}
\def\I{\mathcal I}
\def\H{\mathcal H}
\def\pr{\mathbb{P}}
\def\prev{\mathbb{P}}
\def\K{\mathcal{K}}
\def\G{\mathcal{G}}

\renewcommand{\bot}{\emptyset}
\renewcommand{\top}{\Omega}

\def\ind{\vbox{\hbox{$\bot$\kern-.6em$\bot$}\kern-.05em}}
\def\dep{\vbox{\hbox{$\top$\kern-.6em$\top$}\kern-.05em}}
\newcommand*{\normally}{\mathrel{\ooalign{$|$\hfil\cr\kern+1pt$\thicksim$}}} 
\newcommand*{\nnormally}{\mathrel{\ooalign{$|$\hfil\cr\kern+1pt$\thicksim$}\negthickspace \negthickspace /} } 

\def\I{\mathcal{I}}

\def\F{\mathcal{F}}

\def\J{J}

\title{A Generalized Probabilistic Version of Modus Ponens}
\author{
Giuseppe Sanfilippo\inst{1}\thanks{Partially supported by supported by the INdAM--GNAMPA Project 2016 Grant U 2016/000391}	 \and Niki Pfeifer\thanks{Supported by his DFG project PF~740/2-2 (within the SPP1516)} \inst{2} \and Angelo Gilio\inst{3} \thanks{Retired} }
\institute{
	Department of Mathematics and Computer Science,
	University of Palermo, Italy  
		 \email{giuseppe.sanfilippo@unipa.it}\and 
	Munich Center for Mathematical Philosophy, LMU Munich, Germany
 \email{niki.pfeifer@lmu.de} \and
	Department SBAI,
	University of Rome ``La Sapienza'', Italy
     \email{angelo.gilio@sbai.uniroma1.it}  
}
\renewcommand\top{\Omega}
\renewcommand\bot{\emptyset}
\begin{document}
	\maketitle
\begin{abstract}
Modus ponens (\emph{from $A$ and ``if $A$ then $C$'' infer $C$}; short: MP) is one of the most basic inference rules.  The probabilistic MP allows for managing uncertainty by transmitting assigned uncertainties from the premises to the conclusion (i.e., \emph{from $P(A)$ and $P(C|A)$ infer $P(C)$}). In this paper, we generalize the probabilistic MP by replacing $A$ by the conditional event $A|H$. The resulting inference rule involves  iterated conditionals (formalized by conditional random quantities) and propagates previsions from the premises to the conclusion. Interestingly, the propagation rules for the lower and the upper bounds  on the conclusion of the generalized probabilistic MP  coincide with the respective bounds on the conclusion for  the (non-nested) probabilistic MP.    
\end{abstract}
\section{Introduction}
Modus ponens (\emph{from $A$ and ``if $A$ then $C$'' infer $C$}) is one of the most basic and important inference rules. By instantiating the antecedent of a conditional it allows for detaching the consequent of the conclusion. 
It is well-known that modus ponens is logically valid (i.e., it is impossible that $A$ and $\no{A}\vee C$ are true while $C$ is false, where the event $\no{A}\vee C$ denotes  the material conditional as defined in classical logic).  
It is also well-known that there are philosophical arguments \cite{adams98,edginton14} and psychological arguments  \cite{evans03,pfeifer13b} in favor of the hypothesis that a conditional \emph{if $A$, then $C$} is best represented by a suitable conditional probability assertion $P(C|A)$ and not by a probability of a corresponding material conditional $P(\no{A}\vee C)$.  
Consequently, coherence-based probability logic  generalizes the classical modus ponens probabilistically by  propagating assigned probabilities  from the premises to the conclusion as follows (see, e.g., \cite{pfeifer06d,pfeifer09b,wagner04}): 
\begin{description}
\item[Probabilistic modus ponens] \emph{From $P(A)=x$  \emph{(probabilistic categorical premise)} and $P(C|A)=y$ \emph{(probabilistic conditional premise)} infer\linebreak $xy\leq P(C)\leq xy+1-y$ \emph{(probabilistic conclusion)}}.
\end{description}
In our paper,   $P(C|A)$ is  the probability of the \textit{conditional event} $C|A$  (see, e.g., \cite{coletti02,definetti36,definetti37,gilio16,lad96,PS16}).
The probabilistic modus ponens is $p$-valid (i.e., the premise set  $\{A, C|A\}$ $p$-entails  the conclusion $C$) and probabilistically informative \cite{GOPSsubm,GiSa13IJAR,gilio16,pfeifer09b}. 

In this paper we generalize the probabilistic modus ponens by 
  replacing the categorical premise (i.e., $A$) and the antecedent of the conditional premise (i.e., $A$ in ``\emph{if $A$ then $C$}'') by the conditional event $A|H$. 
  The resulting inference rule involves the prevision $\prev(C|(A|H))$ of the iterated conditional $C|(A|H)$ (formalized by a suitable conditional random quantity, see   \cite{GOPS17,GiSa13c,GiSa13a,GiSa14}) 
   and propagates the uncertainty  from the premises to the conclusion:
 \begin{description}
\item[Generalized probabilistic modus ponens] 
\emph{From $P(A|H)$ \emph{(generalized categorical premise)} and $\prev(C|(A|H))$ \emph{(generalized conditional premise)} infer $P(C)$} (conclusion). 
\end{description}
The conditional event $A|H$ is interpreted as a conditional random quantity, with   $\prev(A|H)=P(A|H)$  (see below). As mentioned above, modus ponens instantiates the antecedent of a conditional and governs the detachment of the consequent of the conclusion. In our generalization, we study the case where the unconditional event $A$ is replaced by the conditional event ($A|H$) and the conditional event $C|A$ is replaced by the iterated conditional $C|(A|H)$. This corresponds to a common-sense reasoning context where instead of a fact $A$ a rule $A|H$ is learned and used for a modus ponens inference.
\\
The outline of the paper is as follows. In Section~\ref{SEC:Preliminary} we first recall basic notions and results on coherence and  previsions of conditional random quantities. Then, we illustrate the notions of conjunction between conditional events and  of iterated conditional, by recalling some results.
In  Section~\ref{SEC:RESULTS} we prove a generalized decomposition formula for conditional events, with other results on compounded and iterated conditionals. In Section~\ref{SEC:generalizedMP} we propagate the previsions from the premises of the generalized probabilistic modus ponens to the conclusion.  We observe that this propagation rule coincides with the probability propagation rule for the (non-nested) probabilistic modus ponens (where $H=\top$) \cite{pfeifer09b}. Section~\ref{SEC:Conclud} concludes the paper with an outlook for future work.
\section{Preliminary notions}\label{SEC:Preliminary}
In this section we recall some basic notions and results on
coherence for conditional prevision assessments. In our approach an event $A$ represents an uncertain fact described by a (non-ambiguous) logical entity, where $A$ is  two-valued and  can be true ($T$), or false ($F$).
The indicator of $A$, denoted by the same symbol, is a two-valued numerical quantity which is 1, or 0, according to  whether $A$ is true, or false, respectively. The sure event is denoted by $\top$ and the impossible event is denoted by $\bot$.  Moreover, we denote by $A\land B$, or simply $AB$, (resp., $A \vee B$) the logical conjunction (resp., logical disjunction).   The  negation of $A$ is denoted by $\no{A}$. Given any events $A$ and $B$, we simply write $A \subseteq B$ to denote that $A$ logically implies $B$, that is,  $A\no{B}$ is the impossible event $\bot$. We recall that  $n$ events are logically independent when the number $m$ of constituents, or possible worlds, generated by them  is $2^n$ (in general  $m\leq 2^n$).
 Given two events $A$ and $H$, with $H\neq\bot$, the
 \emph{conditional event} $A|H$ is defined as a three-valued logical
 entity which is \emph{true} if $AH$ is true,
 \emph{false} if $\no{A}H$ is true, and \emph{void} if $H$ is false. 
 
\subsection{Coherent conditional prevision}
\label{Coherence}
We recall below the notion of coherence (see, e.g., \cite{BiGS08,BiGS12,CaLS07,coletti02,gilio16,GiSa14,PeVa17}).
Given a prevision function $\pr$ defined on an arbitrary family $\K$ of
conditional random quantities 
with finite sets of possible values,  consider a finite subfamily $\F_n = \{X_i|H_i, \, i
\in J_n\} \subseteq \K$, where  $J_n=\{1,\ldots,n\}$, and the vector
$\M_n=(\mu_i, \, i \in J_n)$, where $\mu_i = \pr(X_i|H_i)$ is the
assessed prevision for the c.r.q. $X_i|H_i$.
With the pair $(\F_n,\M_n)$ we associate the random gain $G =
\sum_{i \in J_n}s_iH_i(X_i - \mu_i)$; moreover, we set $\H_n = H_1 \vee \cdots \vee H_n$ and we denote by $\G_{\mathcal{H}_n}$ the set of values of $G$ restricted to $\H_n$. Then, using the {\em betting scheme} of de Finetti, we stipulate
\begin{definition}\label{COER-RQ}{\rm
		The function $\pr$ defined on $\K$ is coherent if and only if, $\forall n
		\geq 1$,  $\forall \, \F_n \subseteq \K,\, \forall \, s_1, \ldots,
		s_n \in \mathbb{R}$, it holds that: $min \; \G_{\mathcal{H}_n} \; \leq 0 \leq max \;
		\G_{\mathcal{H}_n}$. }
\end{definition}
Given a family $\F_n = \{X_1|H_1,\ldots,X_n|H_n\}$, for each $i \in J_n$ we denote by $\{x_{i1}, \ldots,x_{ir_i}\}$ the set of possible values for the restriction of $X_i$ to $H_i$; then, for each $i \in J_n$ and $j = 1, \ldots, r_i$, we set $A_{ij} = (X_i = x_{ij})$. Of course, for each $i \in J_n$, the family $\{\no{H}_i, A_{ij}H_i \,,\; j = 1, \ldots, r_i\}$ is a partition of the sure event $\Omega$, with  $A_{ij}H_i=A_{ij}$, $\bigvee_{j=1}^{r_i}A_{ij}=H_i$. Then,
the constituents generated by the family $\F_n$ are (the
elements of the partition of $\Omega$) obtained by expanding the
expression $\bigwedge_{i \in J_n}(A_{i1} \vee \cdots \vee A_{ir_i} \vee
\no{H}_i)$. We set $C_0 = \no{H}_1 \cdots \no{H}_n$ (it may be $C_0 = \bot$);
moreover, we denote by $C_1, \ldots, C_m$ the constituents
contained in $\H_n = H_1 \vee \cdots \vee H_n$. Hence
$\bigwedge_{i \in J_n}(A_{i1} \vee \cdots \vee A_{ir_i} \vee
\no{H}_i) = \bigvee_{h = 0}^m C_h$.
With each $C_h,\, h \in J_ m$, we associate a vector
$Q_h=(q_{h1},\ldots,q_{hn})$, where $q_{hi}=x_{ij}$ if $C_h \subseteq
A_{ij},\, j=1,\ldots,r_i$, while $q_{hi}=\mu_i$ if $C_h \subseteq \no{H}_i$;
 $C_0$ is associated with  $Q_0=\M_n = (\mu_1,\ldots,\mu_n)$. 
Denoting by $\I_n$ the convex hull of $Q_1, \ldots, Q_m$, the condition  $\M_n\in \I_n$ amounts to the existence of a vector $(\lambda_1,\ldots,\lambda_m)$ such that:
$ \sum_{h \in J_m} \lambda_h Q_h = \M_n \,,\; \sum_{h \in J_m} \lambda_h
= 1 \,,\; \lambda_h \geq 0 \,,\; \forall \, h$; in other words, $\M_n\in \I_n$ is equivalent to the solvability of the system $(\Sigma)$, associated with  $(\F_n,\M_n)$,
\begin{equation}\label{SYST-SIGMA}
\begin{array}{l}
(\Sigma) \;\;\;\sum_{h \in J_m} \lambda_h q_{hi} =
\mu_i \,,\; i \in J_n \,; \; \sum_{h \in J_m} \lambda_h = 1 \,;\;
\lambda_h \geq 0 \,,\;  \, h\in J_m \,.
\end{array}
\end{equation}
Given the assessment $\M_n =(\mu_1,\ldots,\mu_n)$ on  $\F_n =
\{X_1|H_1,\ldots,X_n|H_n\}$, let $S$ be the set of solutions $\Lambda = (\lambda_1, \ldots,\lambda_m)$ of system $(\Sigma)$ defined in  (\ref{SYST-SIGMA}).   
Then, the  following theorem can be proved (\cite{BiGS08})
\begin{theorem}\label{SYSTEM-SOLV}{ \rm [{\em Characterization of coherence}].
Given a family of $n$ conditional random quantities $\F =
\{X_1|H_1,\ldots,X_n|H_n\}$ and a vector $\M =
(\mu_1,\ldots,\mu_n)$, the conditional prevision assessment
$\prev(X_1|H_1) = \mu_1 \,,\, \ldots \,,\, \prev(X_n|H_n) =
\mu_n$ is coherent if and only if, for every subset $J \subseteq J_n$,
defining $\F_J = \{X_i|H_i \,,\, i \in J\}$, $\M_J = (\mu_i \,,\,
i \in J)$, the system $(\Sigma_J)$ associated with the pair
$(\F_J,\M_J)$ is solvable. }\end{theorem}
By following the approach given in \cite{GiSa13c,GiSa13a,GiSa14}  a conditional random quantity $X|H$ can be seen as the random quantity $XH+\mu\no{H}$, where $\mu=\prev(X|H)$. In particular a conditional event $A|H$ can be interpreted as $AH+x\no{H}$, where $x=P(A|H)$.   Moreover, the negation of $A|H$ is defined as  $\no{A|H}=1-A|H=\no{A}|H$. Coherence can be characterized in terms of proper scoring rules (\cite{BiGS12}), which can be related  to the notion of entropy in information theory (\cite{LSA15}). 
\subsection{Conjunction and iterated conditional}
\begin{definition}\label{CONJUNCTION}{\rm Given any pair of conditional events $A|H$ and $B|K$, with $P(A|H)=x$, $P(B|K)=y$, we define their conjunction
		as the conditional random quantity $(A|H) \wedge (B|K) = Z\,|\, (H \vee K)$, where $Z=\min \, \{A|H, B|K\}$.
}\end{definition}
Based on the betting scheme, the compound conditional $(A|H) \wedge (B|K)$ coincides with 
$1 \cdot AHBK + x \cdot \no{H}BK + y \cdot AH\no{K} + z \cdot \no{H}\no{K}$,
where $z$ is the \emph{prevision} of the random quantity $(A|H) \wedge (B|K)$, denoted by $\mathbb{P}[(A|H) \wedge (B|K)]$. Notice that $z$ represents the amount you agree to pay, with the proviso that you will receive the quantity $(A|H) \wedge (B|K)$.  
For examples see \cite{GiSa13a} and \cite{Kauf09}.  Notice that this notion of conjunction, with positive probabilities for the conditioning events, has been already proposed in \cite{McGe89}. Now, we recall the notion of iterated conditioning.
\begin{definition}[Iterated conditioning]\label{ITER-COND}{\rm
Given any pair of conditional events $A|H$ and $B|K$, the iterated conditional $(B|K)|(A|H)$ is the  conditional random quantity 
$(B|K)|(A|H) = (B|K) \wedge (A|H) + \mu \no A|H$,  where $\mu=\mathbb{P}[(B|K)|(A|H)]$.
}
\end{definition}
Notice that, in the context of betting scheme, $\mu$ represents the amount you agree to pay, with the proviso that you will receive the quantity 
\begin{equation}
\small
(B|K)|(A|H)=\left\{\begin{array}{ll}
1, &\mbox{ if }  AHBK \mbox{ true,}\\
0, &\mbox{ if }  AH\no{B}K \mbox{ true,}\\
y, &\mbox{ if }  AH\no{K} \mbox{ true,}\\
\mu, &\mbox{ if }  \no{A}H\mbox{ true,}\\
x+\mu(1-x), &\mbox{ if }  \no{H}BK \mbox{ true,}\\
\mu(1-x), &\mbox{ if }  \no{H}\no{B}K \mbox{ true,}\\
z+\mu(1-x), &\mbox{ if }  \no{H}\no{K} \mbox{ true.}\\
\end{array}
\right.
\end{equation}
We recall the following product formula (\cite{GiSa13a}) 
\begin{theorem}[Product formula]\label{THM:PRODUCT}{\rm
Given any  assessment $x=P(A|H), \mu=\mathbb{P}[(B|K) | (A|H)]$, $z=\mathbb{P}[(B|K) \wedge (A|H)]$, if $(x,\mu,z)$ is coherent, then
 $z=\mu \cdot x$.
}
\end{theorem}
We recall that coherence requires that $(x,y,\mu,z)\in[0,1]^4$ (see, e.g., \cite{GOPS17}).
\begin{remark}\label{REM:AHK}
	Given any random quantity  $X$ and any events $H,K$, with $H\subseteq K$, $H\neq \bot$, it holds that (see \cite[Section 3.3]{GiSa14}): $(X|H)|K=X|HK=X|H$. In particular, given any events $A,H,K$,  $H\neq \bot$, it holds that: 
	$(A|H)|(H\vee K)=A|H$.
\end{remark}
\section{Some results on compounded and iterated conditionals}
\label{SEC:RESULTS}
In this section we present a decomposition formula, by also considering a particular case.
Then, we give a result on the coherence of a prevision assessment on $\mathcal{F}=\{A|H,C|(A|H), C|(\no{A}|H)\}$ which will be used in the next section.
\begin{proposition}\label{PROP:DECOMPOSITION}
	Let $A|H,B|K$ be two conditional  events. Then
	\begin{equation}
	B|K=(A|H)\wedge (B|K)+	(\no{A}|H)\wedge (B|K)\,.
	\end{equation}
\end{proposition}
\begin{proof}
Let $(x,y,z_1,z_2)$ be a (coherent) prevision  on  $(A|H,B|K,(A|H)\wedge (B|K),(\no{A}|H)\wedge (B|K)
)$.  Of course, coherence requires that 
$P(\no{A}|H)=1-x$. By Definition~\ref{CONJUNCTION} it holds that
\[
(A|H)\wedge (B|K)=AHBK+x\no{H}BK+y\no{K}AH+z_1\no{H}\no{K}
\]
and
\[
(\no{A}|H)\wedge (B|K)=\no{A}HBK+(1-x)\no{H}BK+y\no{K}\no{A}H+z_2\no{H}\no{K}\,.
\]
Then,
\begin{equation}\label{EQ:SUM}
\begin{array}{ll}
(A|H)\wedge (B|K)+(\no{A}|H)\wedge (B|K)= HBK+\no{H}BK+y\no{K}H+z_1\no{H}\no{K}+z_2\no{H}\no{K}=\\
=BK+y\no{K}H+(z_1+z_2)\no{H}\no{K}.
\end{array}
\end{equation}
Moreover,
\begin{equation}\label{EQ:BgK}
B|K=BK+y\no{K}=BK+y\no{K}H+y\no{H}\no{K}\,.
\end{equation}
From (\ref{EQ:SUM}) and (\ref{EQ:BgK}), when $H\vee K$ is true, it holds that 
\[
(A|H)\wedge (B|K)+(\no{A}|H)\wedge (B|K)=BK+y\no{K}H=B|K\,.
\]
Then, the difference 
$[(A|H)\wedge (B|K)+(\no{A}|H)\wedge (B|K)]-B|K
$
is zero when $H\vee K$  is true. 
Thus,
\[
\begin{small}
\begin{array}{ll}
\prev[((A|H)\wedge (B|K)+(\no{A}|H)\wedge (B|K)-B|K)|(H\vee K)]=\\
\prev[((A|H)\wedge (B|K))|(H\vee K)]+
\prev[((\no{A}|H)\wedge (B|K))|(H\vee K)]-\prev[(B|K)|(H\vee K)]=0.
\end{array}
\end{small}
\]
By Remark~\ref{REM:AHK} it holds that 
$[(A|H)\wedge (B|K)]|(H\vee K)$, $[(\no{A}|H)\wedge (B|K)]|(H\vee K)$, and 
 $(B|K)|(H\vee K)$ coincide with  $(A|H)\wedge (B|K)$, $(\no{A}|H)\wedge (B|K)$, and 
 $B|K$, respectively.
Then,
\[
\begin{array}{ll}
\prev[((A|H)\wedge (B|K)+(\no{A}|H)\wedge (B|K)-B|K)|(H\vee K)]=\\
\prev[(A|H)\wedge (B|K)]+\prev[(\no{A}|H)\wedge (B|K)]-P(B|K)=z_1+z_2-y=0\,.
\\
\end{array}
\]
Therefore,
 $(A|H)\wedge (B|K)+(\no{A}|H)\wedge (B|K)$ and  $B|K$ also coincide 
when $H \vee K$ is false. Thus, $(A|H)\wedge (B|K)+(\no{A}|H)\wedge (B|K)=B|K$.
\end{proof}
\begin{remark}\label{REM:2}
Notice that Proposition~\ref{PROP:DECOMPOSITION} also holds when  there are some logical relations among the events $A,B,H,K$, provided that $H\neq \bot $ and $K\neq \bot $. In particular, if $K=\top$ the proof of Proposition~\ref{PROP:DECOMPOSITION} is simpler because, 
by Definition~\ref{CONJUNCTION},
\[
(A|H)\wedge B=AHB+x\no{H}B,\;\; (\no{A}|H)\wedge B=\no{A}HB+(1-x)\no{H}B\,,
\]
hence
\begin{equation}\label{EQ:DECOMP}
(A|H)\wedge B+(\no{A}|H)\wedge B=HB+\no{H}B=B.
\end{equation}
\end{remark}
\begin{remark}\label{REM:DIN-COHER}
Consider a bet on an iterated conditional $C|(A|H)$, with $H\neq\bot, A\neq \bot$,  $x=P(A|H)$, and $y=\prev[C|(A|H)]$. In this bet, $y$ is the amount that we  pay, while  $C|(A|H)$ is the amount that we receive. Then, in order to check coherence, the bet on $C|(A|H)$ must be called off when  $C|(A|H)$ coincides  with its prevision $y$. We distinguish two cases: $(i)$ $x>0$, $(ii)$ $x=0$.\\
Case $(i)$. By applying Definition  \ref{ITER-COND},  with $B$ replaced by $C$ and $K=\top$, we obtain 
\begin{equation}
\small
C|(A|H)=\left\{\begin{array}{ll}
1, &\mbox{ if }  AHC \mbox{ true,}\\
0, &\mbox{ if }  AH\no{C} \mbox{ true,}\\
y, &\mbox{ if }  \no{A}H\mbox{ true,}\\
x+y(1-x), &\mbox{ if }  \no{H}C \mbox{ true,}\\
y(1-x), &\mbox{ if }  \no{H}\no{C} \mbox{ true.}
\end{array}
\right.
\end{equation}
If $x>0$, $C|(A|H)=y$ when $\no{A}H$ is true, that is,  the bet on $C|(A|H)$ is called off when $\no{A}H$ is true;
then to check coherence   we must only consider the constituents contained in $\no{\no{A}H}=AH\vee \no{H}$.\\
Case $(ii)$. As $x=0$, we have 
\begin{equation}
\small
C|(A|H)=\left\{\begin{array}{ll}
1, &\mbox{ if }  AHC \mbox{ true}\\
0, &\mbox{ if }  AH\no{C} \mbox{ true}\\
y, &\mbox{ if }  \no{A}H\mbox{ true}\\
y, &\mbox{ if }  \no{H}C \mbox{ true}\\
y, &\mbox{ if }  \no{H}\no{C} \mbox{ true}
\end{array}
\right.=\left\{\begin{array}{ll}
1, &\mbox{ if }  AHC \mbox{ true}\\
0, &\mbox{ if }  AH\no{C} \mbox{ true}\\
y, &\mbox{ if }  \no{AH} \mbox{ true}\\
\end{array}
\right.
\end{equation}
Then  $C|(A|H)=C|AH$ (see \cite[Theorem 4]{GiSa14}) and  to check coherence we must only consider the constituents contained in $AH$.\\
We denote by $(x>0)$  an event which is true or false, according to whether $x$ is positive or not. Then, by unifying Case $(i)$ and Case $(ii)$,  the constituents such that  the bet on $C|(A|H)$ is not called off are those contained in $AH\vee\no{H}(x>0)$.

\end{remark}

\begin{theorem}\label{THM:ABOH}
Let three logically independent events $A,C,H$ be given, with $A\neq \bot$, $H\neq \bot$.
The set of all coherent assessments $\mathcal{M}=(x,y,z)$ on $\mathcal{F}=\{A|H,C|(A|H), C|(\no{A}|H)\}$ is the unit cube $[0,1]^3$.
\end{theorem}
\begin{proof}
Coherence requires that $x,y$, and $z$ must be in $[0,1]$ (see, e.g., \cite{GOPS17}). Thus,
$(x,y,z)$ is not coherent when $(x,y,z)\notin[0,1]^3$.\\
 Let $\mathcal{M}=(x,y,z)\in[0,1]^3$ be  a prevision  assessment on $\mathcal{F}=\{A|H,C|(A|H), C|(\no{A}|H)\}$. Based on Theorem~\ref{SYSTEM-SOLV} we prove coherence by showing that 
 for each subset $\J\subseteq \{1,2,3\}$  the system $(\Sigma_J)$ is solvable.
By Definition~\ref{ITER-COND},
 \[
 C|(A|H)=C\wedge (A|H)+y\no{A}|H\,,\;\; C|(\no{A}|H)=C\wedge (\no{A}|H)+zA|H\,.\\
 \]
We start with $\J=\{1,2,3\}$. Based on Remark~\ref{REM:DIN-COHER}, we observe that  
$\H_3=H_1\vee  H_2 \vee H_3=H \vee (AH \vee \no{H}(x>0)\vee(\no{A}H \vee \no{H}(x<1) )=H \vee \no{H} =\top$.
The constituents $C_h$'s (contained in $\H_3$)  and the corresponding points $Q_h$'s associated with $(\mathcal{F},\mathcal{M})$ are:
\[
\begin{array}{ll}
C_1=AHC,\,
C_2=\no{A}HC,\,
C_3= AH\no{C},\,
C_4= \no{A}H\no{C},\,
C_5=\no{H}C,\,
C_6= \no{H}\no{C},
\end{array}
\]
\[
\begin{array}{ll}
 Q_1=(1,1,z),\,
 Q_2=(0,y,1),\,
 Q_3=(1,0,z),\,
 Q_4=(0,y,0),\\
 Q_5=(x,x+y(1-x),(1-x)+xz),\,
 Q_6=(x, y(1-x),xz)\,.\\
\end{array}
\]
We observe that $Q_5=xQ_1+(1-x)Q_2$ and $Q_6=xQ_3+(1-x)Q_4$, so that
for checking the solvability of the system $(\Sigma_{\J})$ it is enough to consider the points
 $Q_1,Q_2,Q_3,Q_4$. The condition 
$(x,y,z)=\sum_{h=1}^4\lambda_h Q_h$, with $\lambda_h\geq 0$ and $\sum_{h=1}^4\lambda_h=1$, is satisfied for every
$(x,y,z)\in[0,1]^3$. Indeed, the system
\[
\left\{
\begin{array}{ll}
x=\lambda_1+\lambda_3\,, \;
y=\lambda_1+\lambda_2y +\lambda_4y\,, \;
z=\lambda_1z+\lambda_2 +\lambda_3z\,, \; \\
\sum_{h=1}^4\lambda_h =1, \;\; \lambda_h\geq 0\,, \; h=1,2,3,4 
\end{array}
\right.
\]
has the non-negative solution 
\[
\left\{
\begin{array}{ll}
\lambda_1=y(1-\lambda_2-\lambda_4)=y(\lambda_1+\lambda_3)=xy\,, \;\lambda_2=z(1-x)\,, \; \\
\lambda_3=(1-y)x\,, \; \lambda_4=1-z+xz-x=  (1-x)(1-z)\,.
\end{array}
\right.
\]
With  $\J=\{1,2\}$ we  associate the pair $(\F_{\J},\M_{\J})$, where $\F_{\J}=\{A|H,C|(A|H)\}$ and $M_{\J}=(x,y)$. By Remark~\ref{REM:DIN-COHER} we notice that $\H_2={H_1\vee H_2=
H \vee (AH\vee \no{H}(x>0))}$ and then we
distinguish two cases: $(i)$ $x>0$, where $\H_2=\top$; $(ii)$ $x=0$, where $\H_2=H$.\\
Case $(i)$.  
The constituents $C_h$'s (contained in $\H_2=\top$) and the corresponding points $Q_h$'s are:
\[
C_1=AHC\,,\;
C_2=\no{A}H \,,\;
C_3= AH\no{C}\,,\;
C_4=\no{H}C\,,\;
C_5= \no{H}\no{C}\,,\;
\]
\[
Q_1=(1,1)\,,\;
Q_2=(0,y)\,,\;
Q_3=(1,0)\,,\;
Q_4=(x,x+y(1-x))\,,\;
Q_5=(x, y(1-x))\,.\;
\]
We observe that $Q_4=xQ_1+(1-x)Q_2$ and $Q_5=xQ_3+(1-x)Q_2$; then we only refer to $Q_1,Q_2,Q_3$.
The condition $(x,y)=\sum_{h=1}^3\lambda_h Q_h$, with $\lambda_h\geq 0$ and $\sum_{h=1}^3\lambda_h=1$, is satisfied for every
$(x,y)\in[0,1]^2$. Indeed, the system
\[
\left\{
\begin{array}{ll}
x=\lambda_1+\lambda_3\,,\;
y=\lambda_1+\lambda_2y\,,\; \\
\sum_{h=1}^3\lambda_h =1; \;\; \lambda_h\geq 0\,, h=1,2,3\;, 
\end{array}
\right.
\]
which can be written as 
\[
\left\{
\begin{array}{ll}
\lambda_1=y(1-\lambda_2)=y(\lambda_1+\lambda_3)=xy\,, \\
\lambda_2=1-xy-(1-y)x=1-x\,,
\lambda_3=(1-y)x\,,\\
\end{array}
\right.
\]
 is solvable.\\
Case $(ii)$. As $x=0$, by Remark~\ref{REM:DIN-COHER}, it holds that  $C|(A|H)=C|AH$. 
The constituents $C_h$'s (contained in $\H_2=H$) and the corresponding points $Q_h$'s are:
$
C_1=AHC\,,\;
C_2=\no{A}H \,,\;
C_3= AH\no{C}\,,\;
$
$
Q_1=(1,1)\,,\;
Q_2=(0,y)\,,\;
Q_3=(1,0)\,.\;
$
The condition $(0,y)=\sum_{h=1}^3\lambda_h Q_h$, with $\lambda_h\geq 0$ and $\sum_{h=1}^3\lambda_h=1$, is satisfied for every
$y\in[0,1]$. Indeed, the assessment $(0,y)$ coincides with $Q_2$ and  the system is solvable, with $\lambda_1=\lambda_3=0$ and $\lambda_2=1$.
\\
With  $\J=\{1,3\}$ we  associate the pair $(\F_{\J},\M_{\J})$, where $\F_{\J}=\{A|H,C|(\no{A}|H)\}$ and $M_{\J}=(x,z)$.
We note that the assessment $\F_{\J}$ on$\M_{\J}$  is equivalent to the assessment $(1-x,z)$ on $(\no{A}|H,C|(\no{A}|H))$. By the same reasoning as for $\J=\{1,2\}$, the system associated with $(1-x,z)$ on $\{\no{A}|H,C|(\no{A}|H)\}$ is solvable. Then, the system associated with $(x,z)$ on $\{A|H,C|(\no{A}|H)\}$ is solvable too.\\
With  $\J=\{2,3\}$ we  associate the pair $(\F_{\J},\M_{\J})$, where $\F_{\J}=\{C|({A}|H),C|(\no{A}|H)\}$ and $M_{\J}=(y,z)$;
by  Remark~\ref{REM:DIN-COHER},   $H_1=(AH\vee \no{H}(x>0)$, $H_2=(\no{A}H\vee \no{H}(x<1)$, so that
$\H_2= (AH\vee \no{H}(x>0)\vee(\no{A}H\vee \no{H}(x<1))=\top$.
The constituents $C_h$'s  (contained in $\H_2=\top$) and the corresponding points $Q_h$'s  are:
$
C_1=AHC\,,\;
C_2=\no{A}HC\,,\;
C_3= AH\no{C}\,,\;
C_4= \no{A}H\no{C}\,,\;
C_5=\no{H}C\,,\;
C_6= \no{H}\no{C}\,,\;$ and 
$Q_1=(1,z)\,,\;
Q_2=(y,1)\,,\;
Q_3=(0,z)\,,\;
Q_4=(y,0)\,,\;\\
Q_5=(x+y(1-x),(1-x)+xz)\,,\;
Q_6=( y(1-x),xz)\,.
$
We observe that $Q_5=xQ_1+(1-x)Q_2$ and $Q_6=xQ_3+(1-x)Q_4$; then  we only refer to the points $Q_1,Q_2,Q_3,Q_4$. The condition 
$(y,z)=\sum_{h=1}^4\lambda_h Q_h$, with $\lambda_h\geq 0$ and $\sum_{h=1}^4\lambda_h=1$, is satisfied for every
$(y,z)\in[0,1]^2$. Indeed, 
$(y,z)=yQ_1+(1-y)Q_3$.
\\
With  $\J=\{1\}$ we  associate the pair $(\F_{\J},\M_{\J})$, where $\F_{\J}=\{A|H\}$ and $M_{\J}=x$;
the assessment $P(A|H)=x$ is coherent for every $x\in[0,1]$; then the system is solvable.\\
With  $\J=\{2\}$ we  associate the pair $(\F_{\J},\M_{\J})$, where $\F_{\J}=\{C|(A|H)\}$ and $M_{\J}=y$.
By Remark~\ref{REM:DIN-COHER} we notice that $\H_1=
AH\vee \no{H}(x>0)$ and then we
distinguish two cases: $(i)$ $x>0$, where $\H_1=AH\vee \no{H}$; $(ii)$ $x=0$, where $\H_1=AH$.\\
Case $(i)$.   The constituents $C_h$'s (contained in $\H_1=AH\vee \no{H}$) and the corresponding points $Q_h$'s are:
$
C_1=AHC\,,\;
C_2= AH\no{C}\,,\;
C_3=\no{H}C\,,\;
C_4= \no{H}\no{C}\,,\;
Q_1=1\,,\;
Q_2=0\,,\;
Q_3=x+y(1-x)\,,\;
Q_4=y(1-x)\,.\;
$
We observe that $y=yQ_1+(1-y)Q_2$, then  the system  $y=\sum_{h=1}^4\lambda_h Q_h$, with $\lambda_h\geq 0$ and $\sum_{h=1}^4\lambda_h=1$, is solvable; indeed a solution is $(\lambda_1,\lambda_2,\lambda_3,\lambda_4)=(y,1-y,0,0)$.\\
Case $(ii)$.   The constituents $C_h$'s (contained in $\H_1=AH$) and the corresponding points $Q_h$'s are:
$C_1=AHC\,,\;
C_2= AH\no{C}\,,\;
Q_1=1\,,\;
Q_2=0\,,\;
$
We observe that $y=yQ_1+(1-y)Q_2$, then the system  $y=\sum_{h=1}^2\lambda_h Q_h$, with $\lambda_h\geq 0$ and $\sum_{h=1}^2\lambda_h=1$, is solvable, with  the unique solution $(\lambda_1,\lambda_2)=(y,1-y)$.\\
With  $\J=\{3\}$ we  associate the pair $(\F_{\J},\M_{\J})$, where $\F_{\J}=\{C|(\no{A}|H)\}$ and $M_{\J}=z$.
In this case the reasoning is the same as for $\J=\{2\}$, with  $A$ replaced by $\no{A}$, with $x$ replaced by $1-x$, and with $y$ replaced by $z$.\\
In conclusion, the assessment  $(x,y,z)$ on $\{A|H,C|(A|H),C|(\no{A}|H)\}$ is coherent for every $(x,y,z)\in[0,1]^3$.
\end{proof}
\section{Generalized Modus Ponens}
\label{SEC:generalizedMP}
We now  generalize the \emph{Modus Ponens} to the case where the first premise  $A$ is replaced by the conditional event $A|H$.
\begin{theorem}\label{THM:MODUSPONENS}
Given any coherent assessment $(x,y)$ on $\{A|H, C|(A|H)\}$, with $A,C,H$ logically independent, with $A\neq \bot$  and $H\neq \bot$,  the extension $z = P(C)$ is coherent if and only if $z\in[z',z'']$, where 
\begin{equation}
z'=xy \;\; \text{ and } \;\; z'' = xy+1-x \,.
	\end{equation}
\end{theorem}
\begin{proof}
We recall that (Theorem~\ref{THM:ABOH}) the assessment $(x,y)$ on $\{A|H, C|(A|H)\}$ is coherent for every $(x,y)\in[0,1]^2$.
From (\ref{EQ:DECOMP}), by the linearity of prevision, and by  Theorem~\ref{THM:PRODUCT}, 
 we obtain
\[
 \begin{array}{lll}
	z&=&P(C)=\prev[(A|H)\wedge C + (\no{A}|H)\wedge C]=\mathbb{P}[(A|H)\wedge C] +\mathbb{P}[(\no{A}|H)\wedge C ]=\\
&=&	P(A|H)\mathbb{P}[C|(A|H)] +P(\no{A}|H)\mathbb{P}[C|(\no{A}|H)]=xy+(1-x)\mathbb{P}[C|(\no{A}|H)]\,.
	\end{array}
\]
From Theorem~\ref{THM:ABOH}, given any coherent assessment $(x,y)$ on $\{A|H, C|(A|H)\}$, the extension $t=\mathbb{P}[C|(\no{A}|H)]$ on $C|(\no{A}|H)$ is coherent for every $t\in[0,1]$. Then, as $z=xy+(1-x)t$, it follows that $z'=xy$ and $z''=xy+1-x$.
\end{proof}
\begin{remark}
We observe that the result given in Theorem~\ref{THM:MODUSPONENS} also holds when $H=\top$, which is well-known (see, e.g., \cite{pfeifer09b}) 
\end{remark}
We notice that Theorem~\ref{THM:MODUSPONENS} can be rewritten as\\
{\bf Theorem 4$'$.}
\emph{
Given any  logically independent events $A,C,H$, with $A\neq \bot$  and $H\neq \bot$,  the set $\Pi$ of all coherent assessments $(x,y,z)$ on
$\{A|H,C|(A|H),C\}$ is }
\begin{equation}
\Pi=\{(x,y,z)\in[0,1]^3: (x,y)\in[0,1]^2, z\in [xy,xy+1-x]\}.
\end{equation}
\section{Concluding Remarks}
\label{SEC:Conclud}
We generalized the probabilistic modus ponens in terms of conditional random quantities in the setting of coherence. Specifically, we  replaced the categorical premise $A$  and the antecedent $A$ of the conditional   premise $C|A$ by the conditional event $A|H$. 
We proved a generalized decomposition formula for conditional events and we gave some results on compound of conditionals and iterated conditionals. We propagated the previsions from the premises of the generalized probabilistic modus ponens to the conclusion.
 Interestingly,   the lower and the upper bounds  on the conclusion of the generalized probabilistic modus ponens  coincide with the respective bounds on the conclusion for  the (non-nested) probabilistic modus ponens. 
In future work we will focus on similar generalizations of other argument forms like the probabilistic modus tollens. Moreover we will study other instantiations to obtain further generalizations, e.g., by also replacing the  consequent $C$ of the conditional   premise $C|A$  and the conclusion $C$ by a conditional event $C|K$.
\\ \ \\
{\bf Acknowledgments.}
We thank  \emph{DFG},  \emph{FMSH}, and \emph{Villa Vigoni} for supporting joint meetings at Villa Vigoni where parts of this work originated (Project: ``Human Rationality: Probabilistic Points of View'').

\end{document}